\documentclass[12pt]{amsart}
\usepackage[english]{babel}
\usepackage{amsmath}
\usepackage{amssymb}
\usepackage{amsthm}
\usepackage{textcomp}
\usepackage{latexsym}
\usepackage{graphicx}

\newtheorem{thm}{Theorem}
\newtheorem{prob}{Problem}
\newtheorem*{rem}{Remark}
\theoremstyle{definition}
\newtheorem{defi}{Definition}

\renewcommand{\bar}{\overline}

\newcommand{\Z}{{\mathbb{Z}}}
\newcommand{\Q}{{\mathbb{Q}}}
\newcommand{\R}{{\mathbb{R}}}
\newcommand{\C}{{\mathbb{C}}}
\newcommand{\T}{{\mathbb{T}}}

\title[Geom. cont. frac. and Anosov diffeo's]
{Geometric Continued Fractions\linebreak as Invariants\linebreak
in the Topological Classification\linebreak 
of Anosov Diffeomorphisms\linebreak of Tori}

\author[Grisha Kolutsky]{Grisha Kolutsky$^*$}

\date{22 June 2009}

\thanks{
\newline $^*$The work exposed here was
partially supported by grants 7-01-00017-a and\linebreak 08-01-00342-a of the
Russian Foundation for Basic Research, by the grant\linebreak No.~NSh-3038.2008.1
of the President of Russia for support of leading scientific schools
and by the Russian Universities grant No.~RNP.2.1.1.5055.}


\address{\rm Grisha Kolutsky\newline \it
Department of the Theory of Dynamical Systems\newline
Faculty of Mechanics and Mathematics\newline
Lomonosov Moscow State University\newline
MSU, GSP, Glavnoe Zdanie, Leninskie Gory\newline
119899 Moscow, Russia\newline
{\rm e-mail:}  kolutsky@mccme.ru\newline}

\begin{document}

\begin{abstract}We show how an object from the combinatorially geometric
version of the analytical number theory, namely geometric continued
fractions, appears in the classical smooth dynamics, namely in the
problem on the topological classification of Anosov diffeomorphisms of
tori.
\end{abstract}

\maketitle
\tableofcontents

\newpage
\section{Introduction. The historical overview}
The problem on topological classification of Anosov diffeomorphisms
of $n$\nobreakdash-di\-me\-n\-sion\-al torus was first considered in
1960s. J.~Franks in 1969 and C.~Manning in 1973 proved that every
Anosov diffeomorphism of $n$-torus, $\T^{n}$ ($n>1$), is
topologically conjugate to a linear hyperbolic automorphism. It is
easy to see that two linear hyperbolic automorphisms are
topologically conjugate if and only if they are conjugate by a
linear automorphism. So initial problem was reduced to the linear
classification of the linear hyperbolic automorphisms of $\T^{n}$.

There were some algebraic attempts to solve the (reduced) algebraic question. We
mention some of them in the Section 3.

In the case $n=2$ a solution of the last problem goes back to Gauss and
Lagrange. A full invariant is a pair --- trace of the linear hyperbolic operator
and the period of a continued fraction for the slope of an eigenvector of the
operator (this statement seems to be well-known in analytical number theory, but
we know only one reference \cite{AKK}). A geometrical interpretation of this
invariant is the geometric continued fractions constructed by Klein (for
historical survey of the subject see \cite{Ka1}).

Our results presented here are linear classifications of the linear hyperbolic
automorphisms of $\T^{n}$ in the case $n>2$ using notations of geometric
continued fractions, namely Klein's and Klein-Voronoi's generalizations
(see the Theorem \ref{t0} and the Theorem \ref{t2}).

Recently many good results on the geometric continued fractions and related
questions were obtained by different scientists mainly due to interest of
V.~I.~Arnold to this subject. See \cite{Ka1} for further references.

\section{Basic notations and posing the problem}

In the definitions of Anosov maps and of topological conjugacy we are strictly
following the classical notations (see \cite{KH}).

Let us recall the principal definition.

\begin{defi}An operator $A\in GL_{n}(\Z)$ is said to be {\it linear hyperbolic
automorphism (of~$\T^{n}$)} if absolute values of all its eigenvalues are not
equal to~$1$.\end{defi}

Next questions are the main topic of our article:

\begin{enumerate}
\item {\it How to define if two Anosov diffeomorphisms $A$, $B$ of $\T^{n}$
are topologically conjugate?}
\item {\it Does there exist an \textbf{effective} algorithm, determining $A$,
$B$ for the existence of the conjugacy?}
\item {\it What is a complete invariant for the topological classification of
Anosov diffeomorphisms on tori?}
\end{enumerate}

\section{From the topological to the linear classification}

The following statement was the main step in the investigation on topological
classification of Anosov maps on tori.

\begin{thm}Every Anosov diffeomorphism $A$ of $\T^{n}$ is topologically
conjugate to a linear hyperbolic automorphism of $\T^{n}$, i.e. to the action
of~$A$ on the homology group $H_{1}(\T^{n})$.\end{thm}

For a short proof see \cite{KH}. Franks proved it in the assumption that the
nonwandering set of $A$ is the whole torus (see \cite{F}). Manning did it in
full generality (see \cite{M}).

\begin{thm}The topological conjugacy of two linear hyperbolic automorphisms
$A$, $B$ of~$\T_{n}$ imply the existence of a linear conjugacy, i.e.

\begin{center}
$\exists C\in GL_{n}(\Z)$, s.t. $AC=CB$.
\end{center}\end{thm}

Its proof can be easily derived from the action of an Anosov diffeomorphism
on~$H_{1}(\T^{n})$.

Therefore the initial problem reduced to the purely algebraic question.

\begin{prob}\label{p1}Given two linear hyperbolic automorphisms
$A$, $B$ of~$\T_{n}$. Find if $\exists C\in\nobreak GL_{n}(\Z)$ s.t.
$A=CBC^{-1}.$\end{prob}

For the problem of a more broad conjugacy: $A \sim B$  if and only if
$B=CAC^{-1}$ for some $C\in\nobreak GL_{n}(\C)$ a necessary and sufficient
condition is that $A$ and $B$ have the same Jordan Normal Form (JNF). But
the condition $C\in\nobreak GL_{n}(\Z)$ is turned an additional requirement.

This was known already to Gauss. Indeed, in the case~$n=2$ Gauss reduced the
question to the question in the theory of binary quadratic forms. The last
question was solved by him. For description of the reduction see \cite{AKK}.

Also there were some algebraic attempts to solve Problem \ref{p1}.
We mention here just two of them. Grunewald in 1976 (see \cite{G})
introduced explicit algorithms, that decide if~$A$, $B\in
GL_{n}(\Q)$ are conjugate by an element~$C\in\nobreak GL_{n}(\Z)$
or~$C\in\nobreak SL_{n}(\Z)$ and, if such~$C$ exists, find one. But
these algorithms are "highly exponential" (their complexity is very
high) and their geometrical meaning is unclear. Recently Karpenkov
suggested to use Hessenberg matrices \cite{Ka2}. His idea is to
follow the Gauss reduction theory, which finds special reduced matrices
in each conjugacy class. He proved that for a fixed class there
exist finitely many reduced Hessenberg matrices. For details see
\cite{Ka2}.

Our main idea is to use the geometrical interpretation of a complete invariant in the case
$n=2$ and to generalize it to the case $n>2$ using notations of geometric continued fractions
(in the sense of Klein and Klein-Voronoi, all necessary definitions will appear in the
Section 5).

\section{Complete solution in the case $n=2$}

Let $A$ be a hyperbolic automorphism of $\T^{2}$, i.e. $A \in GL_{2}(\Z)$ and
$\lambda$, $\mu \in \R$ are its eigenvalues s.t. $|\mu|<1<|\lambda|$. Let $v=(x,y)^T$
be an expanding eigenvector, i.e. $Av=\lambda v$, with slope $x/y=\omega_A$
being a quadratic irrationality.

\begin{thm}[\bf Lagrange] A decomposition of a real number $\omega$ into a continued
fraction is periodical from some place if and only if $\omega$ is a quadratic
irrationality.\end{thm}

According to this classical result the continued fraction expansion of $\omega_A$ is
periodic:
\begin{multline}\label{Decomp}
\omega_A = [a_0;a_1,a_2,\ldots,a_k ,\underline{a_{k + 1},\ldots,a_{k + q}},
\underline{a_{k + q + 1},\ldots,a_{k + 2q}},\ldots]={}\\
{}=[a_0;a_1,a_2,\ldots,a_k,(a_{k + 1},\ldots,a_{k + q})],
\end{multline}
where $a_{k + iq + j} = a_{k + j}$ for $i \geq 0, \ j = 1,\ldots,q$.

By "the period" of this continued fraction we shall mean not only $q$, but also
the finite sequence of numbers $(a_{k + 1},\ldots,a_{k + q})$ up to a cyclic
permutation.

In \cite{AKK} authors using purely analytic means proved the following statement.

\begin{thm}\label{key} Let $A$ and $B$ be two hyperbolic automorphisms of $\T^{2}$
with the same JNF, then $A$ is conjugate to $B$ via some $C \in GL_{2}(\Z)$ if and only if
continued fraction expansions of $\omega_A$ and $\omega_B$ have the same period
(i.e. the same periodic part).\end{thm}

The statement of the Theorem \ref{key} has a very clear and beautiful geometric
interpretation in terms of geometric continued fractions.

\section{Notations of geometric continued fractions}

In this Section we follow Karpenkov's notations (see \cite{Ka1},
\cite{Ka2} and \cite{Ka3}).

A point of $\R^{n+1}$ is said to
be {\it integer} if all its coordinates are integers. Two sets are called {\it
integer-affine (integer-linearly) equivalent} if there exists an affine (linear)
transformation of $\R^{n+1}$ preserving the lattice of all integer points, and
transforming the first set to the second. A plane is called {\it integer} if it is
integer-affine equivalent to some plane passing through the origin and containing
the sublattice of the integer lattice, and the rank of the sublattice is equal to
the dimension of the plane. A polyhedron (in particular, a triangle or a planar
convex polygon) is said to be {\it integer} if all its vertices are integers.

A segment $PQ$ is said to be {\it integer} if its endpoints ($P$ and $Q$) are
integer. An {\it integer length} of an integer segment $PQ$ is the number of
integer points contained in the interior of the segment plus one. We denote the
integer length by $Il(PQ)$.

Let $P$, $Q$ and $R$ be three integer points that do not lie in the
same straight line. We denote the angle with the vertex at $Q$ and
the rays $QP$ and $QR$ by $\angle PQR$. Also we denote the Euclidean
area of a triangle $PQR$ by $S_{PQR}$.

\begin{defi}Consider an arbitrary integer triangle $PQR$. Then the {\it integer
sine}, $I\sin(PQR)$ of the angle $PQR$ is defined by the following
formula:
$$I\sin(PQR)=\frac{2S_{PQR}}{Il(PQ)Il(QR)}.$$
\end{defi}

It is easy to check that the integer sine is a positive integer-valued function.

Consider arbitrary $n+1$ hyperplanes in $\R^{n+1}$ that intersect at a unique
point, namely the origin. The complement to the union of these hyperplanes
consists of $2^{n+1}$ open orthants. Let us choose an arbitrary orthant.

\begin{defi}The boundary of the convex hull of all integer points except the
origin in the closure of the orthant is called the {\it sail} of the orthant. The
set of all $2^{n+1}$ sails is called the {\it $n$-dimensional continued fraction}
corresponding to the given $n+1$ hyperplanes.\end{defi}

Two $n$-dimensional continued fractions are said to be {\it equivalent} if the
union of all sails of the first continued fraction is integer-linear equivalent to
the union of all sails of the second continued fraction.

\begin{defi}An operator in the group~$SL_{n+1}(\Z)$ is called an {\it integer
irreducible hyperbolic} if the following conditions holds:
\begin{enumerate}
\item {the characteristic polynomial of this operator is irreducible over $\Q$;}
\item {all its eigenvalues are distinct and real.}
\end{enumerate}\end{defi}

Consider some integer irreducible hyperbolic operator~$A\in SL_{n+1}(\Z)$. Let us
take the $n$-dimensional spaces that span all subsets of $n$ linearly independent
eigenvectors of the operator $A$. The spans of every $n$ eigenvectors uniquely
define $n+1$ hyperplanes passing through the origin in general position. These
hyperplanes uniquely define the {\it $n$-dimensional (multidimensional) continued
fraction associated with $A$}.

Let $A$ be an integer irreducible hyperbolic operator. Denote by $\Xi(A)$ the set
of all integer operators commuting with $A$. These operators form a ring with a
standard matrix addition and multiplication. Consider the subset of the set
$SL_{n+1}(\Z)\cap\Xi(A)$ that consists of all operators with positive real
eigenvalues and denote it by $\bar\Xi(A)$. From the Dirichlet unit element
theorem (see \cite{BS}) it follows that the subset $\bar\Xi(A)$ forms a
multiplicative Abelian group isomorphic to $\Z^n$, and that its action is free.
Any operator of this group preserves the integer lattice and the union of all
$n+1$ hyperplanes, and hence it takes the $n$-dimensional continued fraction onto
itself bijectively. (Whenever all eigenvalues are positive, the sails are also
taken onto themselves in a one-to-one way.)

The quotient of a sail under this group action is isomorphic to an $n$-dimensional
torus. By a {\it fundamental domain} we mean the union of some faces that contains
exactly one face from each equivalence class (with respect to the action of the
group $\bar\Xi(A)$).

Previous part of this Section was Klein's version of geometric continued
fractions. Now we turn to a Klein-Voronoi's version.

Consider any real operator $A$ of $SL_n(\R)$ whose eigenvalues are all
distinct. Suppose, that it has $k$ real eigenvalues $r_1,\ldots, r_k$
and $2l$ complex conjugate roots $c_1,\bar c_1,\ldots, c_l, \bar
c_l$, here $k+2l=n$.

Denote by $T^l(A)$ the set of all real operators commuting with
$A$ such that their real eigenvalues are all unit and the absolute
values for all complex eigenvalues equal one. Actually, $T^l(A)$
is an Abelian group with operation of matrix multiplication.

For a vector $v$ in $\R^n$ we denote by $T_A(v)$ the orbit of $v$
with respect of the action of the group of operators $T^l(A)$. If
$v$ is in general position with respect to the operator $A$ (i.e.
it does not lie in invariant planes of $A$), then $T_A(v)$ is
homeomorphic to the $l$-dimensional torus. For a vector of an
invariant plane of $A$ the orbit $T_A(v)$ is also homeomorphic to
a torus of positive dimension not greater than $l$, or to a point.

For instance, if $v$ is a real eigenvector, then $T_A(v)=\{v\}$.
The second example: if $v$ is in a real hyperplane spanned by two
complex conjugate eigenvectors, then $T_A(v)$ is an ellipse.

Let $g_i$ be a real eigenvector with eigenvalue $r_i$
for $i=1,\ldots, k$; $g_{k+2j-1}$ and $g_{k+2j}$ be vectors
corresponding to the real and imaginary parts of some complex
eigenvector with eigenvalue $c_j$ for $j=1,\ldots, l$. We
consider the coordinate system corresponding to the basis
$\{g_i\}$:
$$
OX_1X_2\ldots X_{k}Y_{1}Z_{1}Y_2Z_2\ldots Y_l Z_l.
$$

Denote by $\pi$ the $(k{+}l)$-dimensional plane $OX_1X_2\ldots X_k
Y_1Y_2\ldots Y_l$. Let $\pi_+$ be the cone in the plane $\pi$
defined by the equations $y_i\ge 0$ for $i=1,\ldots, l$. For any
$v$ the orbit $T_A(v)$ intersects the cone $\pi_+$ in a unique
point.

\begin{defi}
A point $p$ in the cone $\pi_+$ is said to be {\it $\pi$-integer}
if the orbit $T_A(p)$ contains at least one integer point.
\end{defi}

Consider all (real) hyperplanes invariant under the action of the
operator~$A$. There are exactly $k$ such hyperplanes. In the
above coordinates the $i$-th of them is defined by the equation
$x_i=0$.

The complement to the union of all invariant hyperplanes in the
cone $\pi_+$ consists of $2^k$ connected components.
Consider one of them.
\begin{defi}
The convex hull of all $\pi$-integer points except the origin
contained in the given connected component is called a
{\it factor-sail} of the operator~$A$. The set of all
factor-sails is said to be the {\it factor-continued fraction}
for the operator~$A$.
\\
The union of all orbits~$T_A(*)$ in~$\R^n$ represented by the
points in the factor-sail is called the {\it sail} of the
operator~$A$. The set of all sails is said to be the  {\it
continued fraction} for the operator~$A$ (in the sense of
Klein-Voronoi).
\end{defi}

The intersection of the factor-sail with a hyperplane in $\pi$ is
said to be an $m$-dimensional face of the factor-sail if it is
homeomorphic to the $m$-dimensional disc.

The union of all orbits in $\R^n$ represented by points in some
face of the factor-sail is called the {\it orbit-face} of the
operator $A$.

Integer points of the sail are said to be {\it vertices} of this
sail.

\begin{defi}Continued fractions for operators $A$ and $B$ (in the
sense of Klein-Voronoi) are said to be {\it equivalent} if the
union of all sails of the first continued fraction is
integer-linear equivalent to the union of all sails of the second
continued fraction\end{defi}

Consider now an operator $A$ in the group $SL_n(\Z)$ with
irreducible characteristic polynomial. Suppose, that it has $k$
real roots $r_1,\ldots, r_k$ and $2l$ complex conjugate roots:
$c_1,\bar c_1,\ldots, c_l, \bar c_l$, where $k+2l=n$. In the
simplest possible cases $k{+}l=1$ any factor-sail of $A$ is a
point. If $k{+}l>1$, than any factor-sail of $A$ is an infinite
polyhedral surface homeomorphic to $\R^{k+l-1}$.

\begin{defi}
The group of all $SL_n(\Z)$-operators commuting with $A$  with
positive eigenvectors is called {\it the Dirichlet group} and
denoted by $\hat\Xi(A)$.
\end{defi}

The Dirichlet group $\hat\Xi(A)$ takes
any sail of $A$ to itself. By Dirichlet unit theorem the group
$\hat\Xi(A)$ is homomorphic to $\Z^{k+l-1}$ and its action on any
sail is free. The quotient of a sail by the action of $\hat\Xi(A)$ is
homeomorphic to the $(n{-}1)$-dimensional torus. By a {\it
fundamental domain} of the sail we mean a collection of open
orbit-faces such that for any $\hat\Xi(A)$-orbit of orbit-faces of
the sail there exists a unique representative in the collection.

\section{Invariants in the case $n>2$}

The geometric interpretation the Theorem \ref{key} is the following (see
\cite{Ka4} for details). A sail for a linear hyperbolic automorphism $A$ of
$\T^2$ is an infinite polygonal chain $\ldots,V_{-2},V_{-1},V_0,V_1,V_2,\ldots$,
where $V_i$ are integer vertices of the sail.

\begin{defi}The infinite sequence of positive integers
$$
(\ldots,Il(V_{-2}V_{-1}),L\sin(\angle
V_{-2}V_{-1}V_0),Il(V_{-1}V_{0}), L\sin(\angle
V_{-1}V_{0}V_{1}),Il(V_{0}V_{1}),\ldots)
$$
is called the {\it LLS-sequence} of the sail.\end{defi}

Let $\omega_A=[a_0;a_1,a_2,\ldots,a_k,(a_{k + 1},\ldots,a_{k + q})]$ be the same as
in~$(\ref{Decomp})$. Then the LLS-sequence of a sail of $A$ is periodic and its
period is equal to $[(a_{k + 1},\ldots,a_{k + q})]$ up to a cyclic permutation
(surely, LLS-sequences of all four sails of $A$ coincide). For the proof see \cite{Ka4}.

Let us generalize the last statement (it would be a multidimensional analogue of the
Theorem \ref{key}) "automatically" by using notations of Klein's geometric continued
fractions. Unfortunately, its definition restricts us from a linear hyperbolic
automorphism to an integer irreducible hyperbolic operators.

Consider two integer irreducible hyperbolic operators~$A,B\in SL_{n}(\Z)$
(automorphisms of $\T^{n}$).

\begin{thm}\label{t0}
Automorphisms $A$ and $B$ are linearly conjugate if and
only if they have the same JNF and their
$(n-1)$\nobreakdash-di\-me\-n\-sion\-al continued fractions are equivalent.\end{thm}

\begin{rem} For $n=2$ the equivalence of two continued fractions means the same
LLS-sequence in geometric terms or the coincidence of their periods (up to a cyclic
permutation) in the analytic language.\end{rem}

\begin{proof}[Proof of the Theorem \ref{t0}]\label{t1}
If $A$ and $B$ are linearly conjugate by some operator
$C\in GL_{n}(\Z)$, then
$C$ realizes the integer-linear equivalence of $(n-1)$-dimensional continued fractions
associated with $A$ and $B$ just by the definition: integer points goes to integer points,
the convex hull in every orthant goes to the convex hull in the new orthant, which is
the image of the old one, and sails goes to sails.

Inversely, $(n-1)$-dimensional continued fractions associated with $A$ and $B$ are
equivalent, then by the definition there exists a linear transformation of $\R^{n}$
preserving the set of all integer points $C$. Naturally, in this case $C$ belongs
to $GL_{n}(\Z)$ and $C$ gives us the sought-for conjugacy.\end{proof}

\begin{rem} In particular all combinatorial data (integer lengths, integer sines,
integer volumes of all dimensions (straightforward generalizations of definitions
of the integer length and the integer sine) and degrees of all vertices in a
fundamental domain) are invariant under an action of a conjugacy.\end{rem}

Of course, assumptions of the Theorem \ref{t0} is sufficiently strong. For
example, the case of nonreal eigenvalues of a hyperbolic automorphism of $\T^{n}$
is not considered. By this reason the Klein-Voronoi's generalization of continued
fractions seems to be more interesting and attractive, than the Klein's
generalization. From the other hand, the Klein-Voronoi's generalization looks
less natural than the Klein's generalization.

Consider two hyperbolic operators~$A,B\in SL_{n}(\Z)$ with irreducible
characteristic polynomials.

\begin{thm}\label{t2}
Operators $A$ and $B$ are linearly conjugate if and only if
they have the same JNF and their
continued fractions (in the sense of Klein-Voronoi) are equivalent.
\end{thm}

\begin{proof}If $A$ and $B$ are linearly conjugate by some operator
$C\in~GL_{n}(\Z)$, then $C$ realizes an integer-linear equivalence of their
continued fractions (in the sense of Klein-Voronoi) just by the definition. Here we
need to check, that the action of a conjugacy well coordinates with the sequence
of definitions of Klein-Voronoi's continued fractions only.

Inversely, continued fractions (in the sense of Klein-Voronoi) associated with $A$
and $B$ are equivalent, then by the definition there exists a linear transformation of
$\R^{n}$ preserving the set of all integer points $C$. Naturally, in this case
$C$ belongs to $GL_{n}(\Z)$ and $C$ gives us the sought-for conjugacy.\end{proof}

\section{Algorithms in the case $n>2$}

We mentioned algorithms of Grunewald \cite{G} in the Section 3, but they have
big complexity. Karpenkov in his Ph.~D.~thesis introduced some "deductive"
algorithms, that constructs a fundamental domain for a given geometric
continued fraction in the sense of Klein. The "deductivity" means here that
at some moment a human should "help" to the algorithm by
looking at an interim result and answer if it is sufficient or not. An
application of this "deductive" algorithms for our problems is straightforward.
Namely, starting from two arbitrary diffeomorphisms $f$, $g$ of $T^{n}$,
one can get two linear Anosov maps $A$ and $B$, which are topologically
conjugate to $f$ and $g$ respectively. If $A$ and $B$ have different JNF than
$f$ and $g$ are not conjugate. Otherwise, we need to compare fundamental domains of $A$
and $B$ for an existence of a conjugacy. It is easy to do it if we had already
constructed these fundamental domains --- there exists an effective algorithm
\cite{Ka1}.

Due to results mentioned above we can formulate a following question.

\begin{prob}Invent an effective algorithm deciding if two hyperbolic
toral automorphisms are linearly conjugate or not and finding a conjugating
matrix if the answer is positive.\end{prob}

\section{Lagrange's theorem and its generalizations}

Here we will discuss a classical Lagrange's theorem about continued
fractions and attempts to generalize it.

In 1994 Korkina \cite{Ko} announced a very natural-looking
generalization of this theorem, but her statement appeared to be
wrong. German is claiming to have constructed a counterexample,
however it was never published. In 2008, German and Lakshtanov
\cite{GL} presented a correct result, which they considered as a
generalization of the Lagrange theorem. The weak points of that
generalization is its unreasonable complexity and the fact that its
geometrical meaning is unclear. So the following question appears
quite naturally.

\begin{prob}
Formulate and prove a simple, natural and correct version of
Lagrange's theorem for geometric continued fractions in the sense of
Klein and Klein-Voronoi.
\end{prob}

\section{Additional invariants in the case of $SL_{n}(\Z)$-clas\-sification}

In the case $n=2$ there is a following difference between classification of
hyperbolic toral automorphisms up to conjugacy by $GL_{2}(\Z)$ or $SL_{2}(\Z)$.

Let $A$ and $B$ be linear hyperbolic automorphisms of $\T^{2}$, which are linear
conjugate by some~$C\in~GL_{2}(\Z)$.

Let $\omega_A=[a_0;a_1,a_2,\ldots,a_k~,(a_{k + 1},\ldots,a_{k + q})]$ be the
same as in~$(\ref{Decomp})$.

In \cite{AKK} authors proved the following statement.

\begin{thm}\begin{enumerate}
\item if $q \not\vdots 2$ then $\exists \tilde C \in SL_{2}(\Z)$: $A\tilde C = \tilde CB$.
\item if $q \vdots 2$, $C \in~SL_{2}(\Z)$ then this $C$ gives a sought-for conjugacy.
\item if $q \vdots 2$, $C \not\in~SL_{2}(\Z)$
then~$\not\exists \tilde C \in~SL_{2}(\Z)$: $A\tilde C = \tilde CB$.
\end{enumerate}\end{thm}

One can conclude from it, that evenness of the size of a period plays an essential
role. We don't know anything about generalizations of the last theorem to the
higher dimensions.

\begin{prob}Find a complete additional invariant, which could distinguish
classification of hyperbolic toral automorphism under the action of~$GL_{n}(\Z)$
and~$SL_{n}(\Z)$ for the case~$n>2$.\end{prob}

\section{Recovery of a hyperbolic automorphism from a period}

One can ask another very natural question, namely the inverse question. What one can
say about the existence of hyperbolic toral automorphisms for a given period? If it
exists, how many nonconjugate types are?

For the case $n=2$ there is an explicit answer, for further references see
\cite{Ka4}. For $n>2$ there is only a theorem of Karpenkov (in his Ph.~D. thesis)
which says, that if for a given fundamental domain provided with combinatorial data
(integer lengths, integer sines, integer volumes of all dimensions, degrees of
the vertices, etc.) such continued fraction in the sense of Klein exists, then it
is unique.

So we can just repeat a classical problem from the theory of geometric
continued fractions.

\begin{prob}Find a necessary and sufficient conditions for a combinatorial data,
that allows us to construct a Klein's geometric continued fraction, which
fundamental domain has exactly this data. The same question for Klein-Voronoi's
continued fractions.\end{prob}

\section{Acknowledgements} The author is grateful to D.~V.~ Anosov for the
scientific advising, to A.~Yu.~Zhirov for posing the problem, to
O.~N.~Karpenkov for the introduction to the beautiful world of geometric
continued fractions and to A.~V.~Klimenko for his friendly assistance.

\end{document}